\documentclass[12pt,a4paper]{article}
\usepackage{latexsym,amssymb,amsfonts,amsmath,amsthm,nccmath,float,enumitem,setspace,bm,authblk,graphicx,arydshln}
\usepackage[usenames,dvipsnames]{xcolor}
\usepackage[hidelinks]{hyperref}
\usepackage[margin=1.86cm]{geometry}
\usepackage{booktabs}

\hypersetup{
    colorlinks,
    linkcolor={blue!80!black},
    citecolor={blue!80!black},
    urlcolor={blue!80!black}
}

\newtheorem{Theorem}{Theorem}[section]
\newtheorem{Proposition}{Proposition}[section]

\newtheorem{Lemma}{Lemma}[section]

\def \leq {\leqslant}
\def \geq {\geqslant}

\let\oldproofname=\proofname
\renewcommand{\proofname}{\rm\bf{\oldproofname}}

\numberwithin{equation}{section}

\allowdisplaybreaks

\begin{document}

\title{Even cycle decomposition thresholds for dense multipartite graphs}

\author[a]{Weiyi Sun}
\author[a,b]{Tao Feng \thanks{Supported by NSFC under Grant 12271023}}
\author[a]{Shikang Yu}
\affil[a]{School of Mathematics and Statistics, Beijing Jiaotong University, Beijing, 100044, P.R. China}
\affil[b]{Hebei Provincial Key Laboratory of Mathematical Theory and Analysis for Network and Data Science, Beijing Jiaotong University, Beijing, 100044, P.R. China}
\renewcommand*{\Affilfont}{\small\it}
\renewcommand\Authands{ and }

\affil[ ]{vesunn@bjtu.edu.cn, tfeng@bjtu.edu.cn, healthyu@bjtu.edu.cn}
\date{}

\maketitle

\begin{abstract}
  Let $r\geq 2$ be an integer. An $r$-partite graph $\Gamma$ with vertex partition $V_1,\ldots,V_r$ is $2$-balanced if there exists a positive integer $n$ such that $n\leq |V_i|\leq 2n$ for every $i\in[r]$. For an integer $\ell\geq3$, let $C_\ell$ denote the cycle of length $\ell$, and define  $\hat{\delta}(\Gamma)=\min\{d_\Gamma(v,V_i)/|V_i|:i\in[r],\ v\in V(\Gamma)\setminus V_i\}$. Let $\hat\delta^r_{C_\ell}$ denote the $C_\ell$-decomposition threshold for $2$-balanced $r$-partite graphs, that is, the least nonnegative real number $\delta$ such that, for every $\varepsilon>0$,
  there exists $n_0$ such that every $C_\ell$-divisible $2$-balanced $r$-partite graph $G$ with $\min_i|V_i|>n_0$ and $\hat\delta(G)\geq\delta+\varepsilon$ admits a $C_\ell$-decomposition. We prove that $\hat\delta^{r}_{C_4}=\frac{2}{3}$ and $\hat\delta^{r}_{C_{2k}}=\frac{1}{2}$ for every $r\geq2$ and every $k\geq3$.
\end{abstract}


\section{Introduction}

All graphs considered here are finite, undirected and simple unless otherwise specified. Let $\Gamma$ be a graph. Denote by $V(\Gamma)$ the vertex set of $\Gamma$ and by $E(\Gamma)$ the edge set of $\Gamma$. Let $\gcd(\Gamma)$ be the greatest common divisor of the degrees of all vertices of $\Gamma$. For $v \in V(\Gamma)$ and $U \subseteq V(\Gamma)$, let $N_\Gamma(v,U):=\{u\in U:uv\in E(\Gamma)\}$, and let $d_\Gamma(v,U)=|N_\Gamma(v,U)|$.

Let $F$ be a subgraph of a graph $G$. The subgraphs of $G$ isomorphic to $F$ are referred to as \emph{copies} of $F$. A classical question in graph theory is whether the edges of $G$ can be partitioned into edge-disjoint copies of $F$; such a partition of $E(G)$ is called an \emph{$F$-decomposition} of $G$. $G$ is said to be \emph{$F$-divisible} if $|E(F)|$ divides $|E(G)|$ and $\gcd(F)$ divides $\gcd(G)$. Kirkman \cite{Kirkman} showed that every triangle-divisible complete graph admits a triangle decomposition. Wilson \cite{Wilson75} generalized this by showing that for any graph $F$, every sufficiently large $F$-divisible complete graph admits an $F$-decomposition.

A further problem is to find $F$-decompositions in $F$-divisible graphs with high minimum degree. In 1970, Nash-Williams \cite{Nash} conjectured that every sufficiently large triangle-divisible graph on $n$ vertices with minimum degree at least $\frac{3}{4}n$ admits a triangle decomposition, and very recently, Delcourt and Postle \cite{DP26} confirmed this conjecture. The \emph{$F$-decomposition threshold} $\delta_F$ is the infimum of all nonnegative real numbers $\delta$ such that, for all sufficiently large integers $n$, every $F$-divisible $n$-vertex graph with minimum degree at least $\delta n$ admits an $F$-decomposition. Equivalently, $\delta_F$ is the least nonnegative real number $\delta$ with the property that for every $\varepsilon>0$, there exists $n_0$ such that every $F$-divisible graph on $n\geq n_0$ vertices with minimum degree at least $(\delta+\varepsilon)n$ admits an $F$-decomposition. A fundamental problem in graph decomposition theory is to determine the $F$-decomposition threshold for every graph $F$ (cf. \cite[Problem~3.2]{GKO}). For an integer $\ell \geq3$, let $C_\ell$ denote the cycle of length $\ell$. Bryant, Dukes, Horsley, Maenhaut, and Montgomery \cite{BDHMM} established the bounds $\frac{1}{2}+\frac{1}{2\ell-2} \leq \delta_{C_\ell}\leq \frac{1}{2}+\frac{1}{2\ell-4}$ for all odd $\ell \geq 5$. Granet and Horsley \cite[Conjecture~1.3]{GH} conjectured that $\delta_{C_\ell}=\frac{1}{2}+\frac{1}{2\ell-2}$ for all odd $\ell \geq 3$, and proved the conjecture for all odd $\ell\geq 73$. For even cycles, Barber, K\"{u}hn, Lo, and Osthus \cite{BKLO} showed that $\delta_{C_4}=\frac{2}{3}$ and $\delta_{C_\ell}=\frac{1}{2}$ for all even $\ell\geq 6$.

This paper focuses on cycle decompositions of multipartite graphs. Let $r\geq 2$ be an integer, and let $[r]:=\{1,2,\ldots,r \}$. Let $\Gamma$ be an $r$-partite graph with vertex partition $V_1,\ldots,V_r$. Let
$$\hat{\delta}(\Gamma):=\min\left\{\frac{d_\Gamma(v,V_i)}{|V_i|}:i\in[r],\ v\in V(\Gamma)\setminus V_i\right\}.$$
We say that $\Gamma$ is \emph{$2$-balanced} if there exists a positive integer $w$ such that $w\leq |V_i|\leq 2w$ for every $i\in[r]$. For any integer $\ell\geq 3$, the \emph{$C_{\ell}$-decomposition threshold for $2$-balanced $r$-partite graphs}, denoted by $\hat\delta^{r}_{C_{\ell}}$, is the least nonnegative real number $\delta$ with the property that for every $\varepsilon>0$, there exists $n_0$ such that every $C_{\ell}$-divisible $2$-balanced $r$-partite graph $G$ satisfying $\min_{i\in[r]}|V_i|>n_0$ and $\hat{\delta}(G)\geq \delta+\varepsilon$ admits a $C_{\ell}$-decomposition.

For any integer $k\geq 2$, let
$$
\alpha_k:=
\begin{cases}
\frac{2}{3}, & k=2;\\[2mm]
\frac{1}{2}, & k\geq 3.
\end{cases}
$$
Taylor \cite[Theorem 1.3 and Proposition 2.2]{Taylor} proved the following result.

\begin{Theorem}{\rm \cite{Taylor}} \label{thm:Taylor}
$\hat\delta^{2}_{C_{2k}} = \alpha_k$ for all integers~$k\geq2$.
\end{Theorem}

In this paper, we determine the $C_{2k}$-decomposition thresholds for $2$-balanced $r$-partite \mbox{graphs}, extending Taylor's result from the bipartite case to arbitrary $r$-partite graphs.

\begin{Theorem}\label{thm:main}
$\hat\delta^{r}_{C_{2k}} = \alpha_k$ for all integers $r,k\geq2$.
\end{Theorem}

\section{Extremal graphs}\label{sec:ex-graph}

In this section, we present extremal constructions that give the lower bound $\hat\delta^{r}_{C_{2k}} \geq \alpha_k$. We first introduce some notation that will be used below.

For disjoint sets $A$ and $B$, let $K_{A,B}$ and $K_A$ denote the complete bipartite graph with bipartition $(A,B)$ and the complete graph on vertex set $A$, respectively. For a graph $\Gamma$, write $\Gamma[U]$ for the subgraph induced by $U\subseteq V(\Gamma)$, and write $\Gamma-H$ for
the graph obtained from $\Gamma$ by deleting the edges of a subgraph $H$. For disjoint subsets $X,Y\subseteq V(\Gamma)$, let
$$E_\Gamma(X,Y):=\{xy\in E(\Gamma):x\in X,y\in Y\}$$ and let $\Gamma[X,Y]$ denote the bipartite subgraph of $\Gamma$ with vertex bipartition $(X,Y)$ and edge set $E_{\Gamma}(X,Y)$.

\begin{Proposition}\label{prop:lower}
Let $r\geq 2$ be an integer.
\begin{enumerate}
  \item[\textup{(1)}] For every sufficiently large even integer $m$, there exists a $C_4$-divisible $r$-partite graph $G$ with vertex partition $V_1,\ldots,V_r$ satisfying $|V_i|=3m$ for every $i\in[r]$, such that
  $$\min\left\{d_G(v,V_i):i\in[r],\ v\in V(G)\setminus V_i\right\}=2m-2,$$
  and $G$ has no $C_4$-decomposition. In particular, $\hat{\delta}(G)=\frac{2m-2}{3m}$.

  \item[\textup{(2)}] Let $k\geq 3$ be an integer. For every sufficiently large even integer $m$, there exists a $C_{2k}$-divisible $r$-partite graph $G$ with vertex partition $V_1,\ldots,V_r$ satisfying $|V_i|=2m$ for every $i\in[r]$, such that
  $$\min\left\{d_G(v,V_i):i\in[r],\ v\in V(G)\setminus V_i\right\}=m-2,$$
  and $G$ has no $C_{2k}$-decomposition. In particular, $\hat{\delta}(G)=\frac{m-2}{2m}$.
\end{enumerate}
Consequently, $\hat\delta^{r}_{C_4}\geq\frac{2}{3}$ and $\hat\delta^{r}_{C_{2k}}\geq\frac{1}{2}$ for every $k\geq 3$.
\end{Proposition}

\begin{proof}
(1) Let $m$ be a sufficiently large even integer. Let $V_1,\dots,V_r$ be disjoint vertex sets of size $3m$. For every $i\in[r]$, partition $V_i$ into $A_i,B_i$, and $C_i$, each of size $m$. Let $H$ be an $r$-partite graph with vertex partition $V_1,\ldots,V_r$, and
$$E(H)=\bigcup_{1\leq i< j\leq r} \left( E(K_{A_i,A_j}) \cup E(K_{A_i,B_j}) \cup E(K_{B_i,A_j}) \cup E(K_{B_i,C_j}) \cup E(K_{C_i,B_j}) \cup E(K_{C_i,C_j}) \right).$$
Choose distinct vertices $a_1,a_1'\in A_1$ and $a_2,a_2'\in A_2$, and choose $b_i\in B_i$ and $c_i\in C_i$, $i=1,2$. Let $Q:=a_1a_2b_1c_2c_1b_2a'_1a'_2a_1$ be a copy of $C_8$ in $H$. Let $G=H-Q$. It is readily checked that the degree of each vertex of $G$ is even and that
$$\min\left\{d_G(v,V_i):i\in[r],\ v\in V(G)\setminus V_i\right\}=2m-2.$$
Moreover, $|E(G)|=|E(H)|-|E(Q)|=6m^2\binom{r}{2}-8$, which is divisible by $4$. Hence, $G$ is $C_4$-divisible.

Let $A=\bigcup_{i=1}^r A_i$, $B=\bigcup_{i=1}^r B_i$, and $C=\bigcup_{i=1}^r C_i$. Let $R$ be an arbitrary copy of $C_4$ in $G$. If $|V(R)\cap A|\in \{0,1\}$, then $|E(R)\cap E(G[A])|=0$. If $|V(R)\cap A|=3$, then $|E(R)\cap E(G[A])|=2$. If $|V(R)\cap A|=4$, then $|E(R)\cap E(G[A])|=4$. If $|V(R)\cap A|=2$, then since $E_G(A,C)=\varnothing$ and $E(G[B])=\varnothing$, the two vertices of $R$ in $A$ cannot be adjacent, and hence $|E(R)\cap E(G[A])|=0$. Therefore, every copy of $C_4$ in $G$ contains an even number of edges in $G[A]$. Since $|E(G[A])|=|E(H[A])|-3=m^2\binom {r}{2}-3$ is odd, there is no $C_4$-decomposition of $G$.

(2) Let $m$ be a sufficiently large even integer so that $2m\geq 2k+2$. Let $V_1,\dots,V_r$ be disjoint vertex sets of size $2m$. For every $i\in[r]$, partition $V_i$ into $A_i$ and $B_i$, each of size $m$. Let $H$ be an $r$-partite graph with vertex partition $V_1,\ldots,V_r$, and
$$E(H)=\bigcup_{1\leq i< j\leq r} ( E(K_{A_i,A_j}) \cup E(K_{B_i,B_j}) ).$$
Choose even integers $a,b$ with $4\leq a,b\leq 2k+2$ such that $a\equiv m^2\binom{r}{2}-2\pmod{2k}$ and $b\equiv m^2\binom{r}{2}+2\pmod{2k}$. Choose a cycle $R_A$ of length $a$ in $K_{A_1,A_2}$ and a cycle $R_B$ of length $b$ in $K_{B_1,B_2}$. Let $G=H-R_A-R_B$. It is readily checked that the degree of each vertex of $G$ is even and that
$$\min\left\{d_G(v,V_i):i\in[r],\ v\in V(G)\setminus V_i\right\}=m-2.$$
Moreover,
$$ |E(G)|=|E(H)|-|E(R_A)|-|E(R_B)| =2m^2\binom{r}{2}-a-b =m^2\binom{r}{2}-a+m^2\binom{r}{2}-b,$$
which is divisible by $2k$. Hence, $G$ is $C_{2k}$-divisible.

Let $A=\bigcup_{i=1}^r A_i$ and $B=\bigcup_{i=1}^r B_i$. By construction, there are no edges between $A$ and $B$, and hence every copy of $C_{2k}$ in $G$ is contained entirely in $G[A]$ or in $G[B]$. Since $|E(G[A])| \equiv 2\pmod{2k}$, the edges of $G[A]$ cannot be partitioned into copies of $C_{2k}$. Therefore, $G$ has no $C_{2k}$-decomposition.
\end{proof}

\section{Proof of Theorem \ref{thm:main}}

The extremal constructions in Section~\ref{sec:ex-graph} give the lower bound $\hat\delta^{r}_{C_{2k}} \geq \alpha_k$. In this section, we prove the matching upper bound $\hat\delta^{r}_{C_{2k}} \leq \alpha_k$.

We first introduce some notation. Let $\Gamma$ be an $r$-partite graph with vertex partition $V_1,\ldots,V_r$, where $r\geq 2$. For $v\in V(\Gamma)$, let $d_\Gamma(v)$ denote the degree of $v$ in
$\Gamma$. For $v\in V(\Gamma)\setminus V_i$, we call $d_\Gamma(v,V_i)$ the \emph{partite degree} of $v$ into $V_i$. The \emph{maximum partite degree} of $\Gamma$ is defined by
$$\hat\Delta(\Gamma):=\max\left\{d_\Gamma(v,V_i):i\in[r],\ v\in V(\Gamma)\setminus V_i\right\}.$$
For graphs $\Gamma_1$ and $\Gamma_2$, write $\Gamma_1\cup\Gamma_2$ for the graph with vertex set $V(\Gamma_1)\cup V(\Gamma_2)$ and edge set $E(\Gamma_1)\cup E(\Gamma_2)$.

We now prove the upper bound. Let $\varepsilon>0$, and let $G$ be a sufficiently large $C_{2k}$-divisible $2$-balanced $r$-partite graph with vertex partition $V_1,\ldots,V_r$ satisfying
$\hat{\delta}(G)\geq\alpha_k+\varepsilon$. The main idea is to remove from $G$ a subgraph $H$ that admits a $C_{2k}$-decomposition so that, for every $1\leq i<j\leq r$, the bipartite graph $(G-H)[V_i,V_j]$ is $C_{2k}$-divisible and satisfies $\hat{\delta}\big((G-H)[V_i,V_j]\big)>\alpha_k+\frac{\varepsilon}{2}$. Theorem \ref{thm:Taylor} then yields a $C_{2k}$-decomposition of each $(G-H)[V_i,V_j]$, and combining these decompositions with a $C_{2k}$-decomposition of $H$ gives a $C_{2k}$-decomposition of $G$.

It remains to explain how we construct $H$. To ensure that each $(G-H)[V_i,V_j]$ is $C_{2k}$-divisible, it suffices to ensure that every partite degree is even and that the number of edges between every pair of vertex parts is divisible by $2k$. We achieve this in three steps. First, Lemma~\ref{lem:part_pair_even} gives a subgraph $H_1$ of $G$ that admits a $C_{2k}$-decomposition such that every pair of vertex parts in $G-H_1$ has an even number of edges. Next, Lemma~\ref{lem:partite_degree_even}, applied to $G-H_1$, gives a subgraph $H_2$ that admits a
$C_{2k}$-decomposition such that every partite degree of $G-H_1-H_2$ is even. Finally, Lemma~\ref{lem:pairwise_C2k_divisible}, applied to $G-H_1-H_2$, gives a subgraph $H_3$ that admits a
$C_{2k}$-decomposition such that every pair of vertex parts in $G-H_1-H_2-H_3$ has a number of edges divisible by $2k$, while all partite degrees remain even. Setting $H:=H_1\cup H_2\cup H_3$, we obtain a graph admitting a $C_{2k}$-decomposition. Moreover, the three lemmas ensure that $\hat\Delta(H)$ is bounded by a constant
independent of the sizes of the vertex parts. Consequently, if the vertex parts of $G$ are sufficiently large, then $\hat{\delta}\big((G-H)[V_i,V_j]\big)>\alpha_k+\frac{\varepsilon}{2}$ for every $1\leq i<j\leq r$, as required.

Subsection~\ref{sec:H1} is devoted to the construction of $H_1$. In Subsection~\ref{sec:H23}, we construct $H_2$ and $H_3$. Finally, Subsection~\ref{sec:main} completes the proof of Theorem~\ref{thm:main}.

\subsection{The construction of $H_1$}\label{sec:H1}

The purpose of this subsection is to prove Lemma~\ref{lem:part_pair_even}, which allows us to remove a $C_{2k}$-decomposable subgraph from $G$ such that, in the remaining graph, the number of edges between every pair of vertex parts is even.

Let $A$ and $B$ be finite sets. The \emph{symmetric difference} of $A$ and $B$ is defined by $A\triangle B:=(A\setminus B)\cup(B\setminus A)$. More generally, if $A_1,\ldots,A_m$ are finite sets, then $A_1\triangle\cdots\triangle A_m$, denoted by $\triangle_{i=1}^m A_i$, consists precisely of those elements that belong to an odd number of the sets $A_1,\ldots,A_m$.

\begin{Lemma}\label{4cycles_sym_diff}
Let $r\geq 4$, and let $Q$ be a spanning subgraph of $K_{[r]}$. Suppose that $|E(Q)| \equiv 0 \pmod 2$ and $d_Q(i)\equiv 0 \pmod 2$ for every $i\in[r]$. Then there exist pairwise distinct $4$-cycles $D_1,\ldots,D_m$ in $K_{[r]}$, with $0\leq m\leq \binom{r-1}{2}$, such that $E(Q)=E(D_1)\triangle\cdots\triangle E(D_m)$.
\end{Lemma}

\begin{proof}
Fix a vertex $j\in[r]$. If every edge of $Q$ is incident with $j$, then $E(Q)=\varnothing$, since any edge $jx\in E(Q)$ would give $d_Q(x)=1$, contradicting the even-degree assumption. In this case,
the conclusion is trivial. Hence, we may assume that $Q$ has an edge $ab$ with $a,b\neq j$.

For each edge $ab\in E(Q)$ with $a,b\neq j$, let $T_{ab}$ be the
triangle on $\{j,a,b\}$. We claim that
\begin{align}
E(Q)= \mathop{\triangle}_{\substack{ab\in E(Q)\\ a,b\neq j}}E(T_{ab}). \label{eq:symdiff1}
\end{align}
Indeed, if $x,y\neq j$, then the edge $xy$ occurs in the right-hand side of \eqref{eq:symdiff1} exactly when $xy\in E(Q)$. On the other hand, for $x\neq j$, the edge $jx$ occurs $d_Q(x)-1$ times if
$jx\in E(Q)$, and $d_Q(x)$ times otherwise. Since $d_Q(x)$ is even,
$jx$ occurs an odd number of times precisely when $jx\in E(Q)$. Thus \eqref{eq:symdiff1} follows.

The number of triangles in \eqref{eq:symdiff1} is $|E(Q)|-d_Q(j)$, which is even. Pair these triangles arbitrarily. If a pair shares a vertex other than $j$, say $T_{xy}$ and $T_{xw}$, then
\[
E(T_{xy})\triangle E(T_{xw})=E(jyxwj),
\]
where $jyxwj$ is a $4$-cycle in $K_{[r]}$. If a pair intersects only at
$j$, say $T_{xy}$ and $T_{uw}$, then
\[
E(T_{xy})\triangle E(T_{uw})
=E(jyxuj)\triangle E(jxuwj),
\]
where both $jyxuj$ and $jxuwj$ are $4$-cycles in $K_{[r]}$. Thus each pair of triangles can be replaced by at most two $4$-cycles without changing the symmetric difference in \eqref{eq:symdiff1}. Consequently,
$E(Q)$ is the symmetric difference of the edge sets of $4$-cycles in $K_{[r]}$.

There are $(|E(Q)|-d_Q(j))/2$ pairs of triangles, and each pair produces at most two $4$-cycles. Hence the number $m$ of resulting $4$-cycles satisfies $m\leq |E(Q)|-d_Q(j)$. Since $|E(Q)|-d_Q(j)$ is the number of edges of $Q$ not incident with $j$, we have $m\leq |E(Q)|-d_Q(j)\leq \binom{r-1}{2}$.
\end{proof}

Let $r\geq4$ and $k\geq2$ be integers, and let $G$ be an $r$-partite graph with vertex partition $V_1,\ldots,V_r$. A \emph{part-index $4$-cycle} is a $4$-cycle in $K_{[r]}$. We write such a cycle as $$D=p_1p_2p_3p_4p_1,$$
where $p_1,p_2,p_3,p_4\in[r]$ are pairwise distinct.
A \emph{$C_{2k}$-lift} of $D$ in $G$ is a copy
$$L = z_1z_2\cdots z_{2k-1}z_{2k}z_1$$
of $C_{2k}$ in $G$ such that
$$
z_a\in V_{p_1}\ (1\leq a\leq 2k-3,\ a\text{ odd}),\quad
z_b\in V_{p_2}\ (2\leq b\leq 2k-2,\ b\text{ even}),\quad
z_{2k-1}\in V_{p_3},\quad
z_{2k}\in V_{p_4}.
$$
In this case, we say that $D$ is \emph{lifted} to $L$. Then $|E_L(V_{p_1},V_{p_2})|=2k-3$ and $|E_L(V_{p_2},V_{p_3})|=|E_L(V_{p_3},V_{p_4})|=|E_L(V_{p_4},V_{p_1})|=1$,
while $L$ contains no edge between any other pair of vertex parts. In particular, deleting $L$ changes the parity of the number of edges between exactly the four pairs of vertex parts corresponding to the edges of $D$.

\begin{Lemma}\label{lift_4cycle_to_C2k}
Let $r\geq 4$ and $k\geq 2$ be integers, and let $\gamma>0$. Then there exists an integer $n_0$ such that the following holds. For every integer $n>n_0$, let $G$ be an $r$-partite graph with vertex partition $V_1,\ldots,V_r$ satisfying $|V_i|\geq n$ for every $i\in[r]$ and $\hat{\delta}(G)\geq \frac{1}{2}+\gamma$. Let $\mathcal D$ be a multiset of part-index $4$-cycles of $G$ such that $|\mathcal D|\leq \binom{r-1}{2}$. Then there exist pairwise vertex-disjoint $C_{2k}$-lifts in $G$, one for each member of $\mathcal D$. Moreover, if $F$ denotes the union of these $C_{2k}$-lifts, then $\hat\Delta(F)\leq 2$.
\end{Lemma}

\begin{proof}
Set $m:=|\mathcal D|$, and fix an ordering $\mathcal D=(D_1,\ldots,D_m)$, with repetitions retained according to their multiplicities. If $m=0$, then the conclusion is trivial. Thus, assume that $m\geq1$. Since $m\leq\binom{r-1}{2}$, by taking $n_0$ sufficiently large, we may assume
that
$$
n>2km,\quad
\left(\frac12+\gamma\right)n>2km
\quad\text{and}\quad
2\gamma n>2km.
$$

We construct the desired lifts greedily. Suppose that, for some $0\leq a\leq m-1$, we have already constructed pairwise vertex-disjoint $C_{2k}$-lifts $L_1,\ldots,L_a$ of $D_1,\ldots,D_a$, respectively. Write $D_{a+1}=p_1p_2p_3p_4p_1$. We construct a $C_{2k}$-lift $L_{a+1}=z_1z_2\cdots z_{2k-1}z_{2k}z_1$
of $D_{a+1}$ which is vertex-disjoint from $L_1,\ldots,L_a$ as follows.

Choose the vertices $z_1,\ldots,z_{2k}$ successively. During the construction, the vertices of $L_1,\ldots,L_a$ and those already chosen for $L_{a+1}$ are forbidden. Since $2ka+2k=2k(a+1)\leq 2km$,
at most $2km$ vertices are forbidden at any stage. Since $|V_{p_1}|\geq n>2km$, we may choose
$z_1\in V_{p_1}$ outside $V(L_1)\cup\cdots\cup V(L_a)$.
Suppose that $z_1,\ldots,z_q$ have been chosen, where
$1\leq q\leq 2k-3$. Define
$$
U_{q+1}:=
\begin{cases}
V_{p_1},&\text{if }z_q\in V_{p_2};\\
V_{p_2},&\text{if }z_q\in V_{p_1}.
\end{cases}
$$
Since $d_G(z_q,U_{q+1})\geq \left(\frac12+\gamma\right)n>2km$, there exists
$z_{q+1}\in N_G(z_q,U_{q+1})$ that is not forbidden. Thus we can choose $z_1,\ldots,z_{2k-2}$ successively. Next, since $d_G(z_{2k-2},V_{p_3})\geq\left(\frac12+\gamma\right)n>2km$,
we may choose $z_{2k-1}\in N_G(z_{2k-2},V_{p_3})$ outside the forbidden vertices. Finally, since
$$
|N_G(z_{2k-1},V_{p_4})\cap N_G(z_1,V_{p_4})|
\geq 2\gamma n>2km,
$$
we may choose $z_{2k}\in N_G(z_{2k-1},V_{p_4})\cap N_G(z_1,V_{p_4})$
outside the forbidden vertices. Hence $L_{a+1}=z_1z_2\cdots z_{2k-1}z_{2k}z_1$ is a $C_{2k}$-lift of $D_{a+1}$ and is vertex-disjoint from $L_1,\ldots,L_a$.

By induction, all members of $\mathcal D$ admit pairwise vertex-disjoint
$C_{2k}$-lifts $L_1,\ldots,L_m$. Let $F=\bigcup_{a=1}^{m}L_a$.
Since the lifts are vertex-disjoint, every vertex of $F$ belongs to at most one lift. As each lift is a copy of $C_{2k}$, every vertex has degree at most $2$ in $F$. Hence
$\hat\Delta(F)\leq2$.
\end{proof}

\begin{Lemma}\label{lem:part_pair_even}
Let $r\geq3$ and $k\geq2$ be integers, and let $\gamma>0$. Then there exists an integer $n_0$ such that the following holds. For every integer $n>n_0$, let $G$ be a $C_{2k}$-divisible $r$-partite graph with vertex partition $V_1,\ldots,V_r$ satisfying $n\leq |V_i|\leq 2n$ for every $i\in[r]$ and $\hat{\delta}(G)\geq\frac12+\gamma$. Then there exists a subgraph $H_1$ of $G$ such that
\begin{enumerate}
\item[\textup{(i)}] $H_1$ admits a $C_{2k}$-decomposition and $\hat\Delta(H_1)\leq 2$;
\item[\textup{(ii)}] $G-H_1$ is $C_{2k}$-divisible;
\item[\textup{(iii)}] $|E_{G-H_1}(V_i,V_j)|\equiv0\pmod2$ for every $1\leq i<j\leq r$.
\end{enumerate}
\end{Lemma}

\begin{proof}
Let $Q$ be the spanning subgraph of $K_{[r]}$ such that $ij\in E(Q)$ if and only if $|E_G(V_i,V_j)|\equiv1\pmod2$. Since $G$ is $C_{2k}$-divisible, every vertex of $G$ has even degree.
Therefore, for every $i\in[r]$,
$$ d_Q(i) \equiv \sum_{j\neq i}|E_G(V_i,V_j)| = \sum_{v\in V_i}d_G(v) \equiv 0\pmod2.$$
Moreover,
$$
|E(Q)|\equiv \sum_{1\leq i<j\leq r}|E_G(V_i,V_j)|=|E(G)|\equiv0\pmod2.
$$

If $r=3$, then $Q$ is a subgraph of $K_3$ with all degrees even. Hence $Q$ is either empty or $K_3$. Since $|E(Q)|$ is even, $Q$ must be empty. Thus,
$|E_G(V_i,V_j)|$ is even for every $1\leq i<j\leq3$, and the conclusion holds with $H_1=\varnothing$. Hence, assume that $r\geq4$. By Lemma~\ref{4cycles_sym_diff}, there exist $4$-cycles $D_1,\ldots,D_m$ in $K_{[r]}$, with $0\leq m\leq\binom{r-1}{2}$, such that
\begin{align}
E(Q)=E(D_1)\triangle\cdots\triangle E(D_m).
\label{eq:symdiff2}
\end{align}
Consequently, every edge $ij\in E(Q)$ is contained in an odd number of the $4$-cycles $D_1,\ldots,D_m$, while every edge $ij\notin E(Q)$ is contained in an even number of them. By Lemma~\ref{lift_4cycle_to_C2k}, these $4$-cycles $D_1,\ldots,D_m$ admit pairwise vertex-disjoint $C_{2k}$-lifts $L_1,\ldots,L_m$ in $G$. Let
$$
H_1:=\bigcup_{\nu=1}^{m}L_\nu .
$$
Since the lifts are pairwise vertex-disjoint and every vertex of each
$L_\nu$ has degree $2$ in $L_\nu$, we have $\hat\Delta(H_1)\leq2$. Moreover, $H_1$ admits a $C_{2k}$-decomposition.

For each $\nu\in[m]$, deleting the edges of the $C_{2k}$-lift $L_\nu$ changes the parity of the number of edges between precisely the four pairs of vertex parts corresponding to the edges of $D_\nu$. Therefore,
by \eqref{eq:symdiff2}, after deleting all of $L_1,\ldots,L_m$, the parity of $|E_G(V_i,V_j)|$ is changed an odd number of times precisely when $ij\in E(Q)$. Hence,
$$
|E_{G-H_1}(V_i,V_j)|\equiv0\pmod2
$$
for every $1\leq i<j\leq r$. Since $H_1$ is a union of edge-disjoint copies of $C_{2k}$, we have
\[
|E(G-H_1)|
\equiv
|E(G)|
\equiv0\pmod{2k}.
\]
Moreover, every vertex loses an even number of incident edges when $H_1$ is deleted. Hence every vertex of $G-H_1$ has even degree, and therefore $G-H_1$ is $C_{2k}$-divisible.

This proves all the required properties for $H_1$, and the proof is complete.
\end{proof}

\subsection{The constructions of $H_2$ and $H_3$}\label{sec:H23}

In this subsection, we construct the subgraphs $H_2$ and $H_3$. We first prove Lemma \ref{lem:realization}, and then apply it to prove Lemma \ref{lem:partite_degree_even} and Lemma \ref{lem:pairwise_C2k_divisible}, which produce $H_2$ and $H_3$, respectively.

Suppose that $r\geq3$ and $k\geq2$, and let $G$ be an $r$-partite graph
with vertex partition $V_1,\ldots,V_r$. A \emph{tuple} of $G$ is an ordered tuple $(p_1;x,y;p_2,p_3)$, where $p_1,p_2,p_3\in[r]$ are pairwise distinct and $x,y\in V_{p_1}$ are distinct vertices. A \emph{$C_{2k}$-realization} of $(p_1;x,y;p_2,p_3)$ is a copy of $C_{2k}$ in $G$ of the form
$$ xz_1z_2\cdots z_{2k-3}yz_{2k-2}x,$$
where
$$
z_a\in V_{p_1}\ (2\leq a\leq 2k-4,\ a\text{ even}),\quad
z_b\in V_{p_2}\ (1\leq b\leq 2k-3,\ b\text{ odd}),\quad
z_{2k-2}\in V_{p_3}.
$$
The vertices $x$ and $y$ are called the \emph{basic vertices}, and the vertices $z_1,\ldots,z_{2k-2}$ are called the \emph{auxiliary vertices}. If $R$ is a $C_{2k}$-realization of $(p_1;x,y;p_2,p_3)$, then $|E_R(V_{p_1},V_{p_2})|=2k-2$ and $|E_R(V_{p_1},V_{p_3})|=2$.

\begin{Lemma}\label{lem:realization}
Let $r\geq 3$ and $k\geq 2$ be integers, let $\gamma>0$, and let $K,K',C\geq0$. Then there exist constants $C'>0$ and $n_0$ such that the following holds. For every integer $n>n_0$, let $G$ be an $r$-partite graph with vertex partition $V_1,\ldots,V_r$ satisfying $|V_i|\geq n$ for each $i\in[r]$ and $\hat{\delta}(G)\geq\frac{1}{2}+\gamma$.
Let $F$ be a subgraph of $G$ with the same vertex partition satisfying $\hat\Delta(F)\leq C$. Let $\mathcal D$ be a multiset of tuples of $G$ such that $|\mathcal D|\leq Kn$, and suppose that every vertex $v\in V(G)$ occurs as a basic vertex in at most $K'$ members of $\mathcal D$. Then there exist pairwise edge-disjoint $C_{2k}$-realizations in $G-F$, one for each member of $\mathcal D$. Moreover, if $F'$ denotes the union of these $C_{2k}$-realizations, then $\hat\Delta(F')\leq C'$.
\end{Lemma}

\begin{proof}
Set $m:=|\mathcal D|\leq Kn$, and fix an ordering $\mathcal D=(\tau_1,\ldots,\tau_m)$, where repetitions are retained according to their multiplicities. If $m=0$, then $\mathcal D=\varnothing$, and the conclusion is immediate. Thus, assume that $m\geq1$.

We first consider the case $k\geq3$. Suppose that, for some $0\leq s\leq m-1$, we have already constructed pairwise edge-disjoint $C_{2k}$-realizations $R_1,\ldots,R_s$ of $\tau_1,\ldots,\tau_s$, respectively, in $G-F$. For $s\geq1$, let
\[
F_s:=\bigcup_{j=1}^{s}R_j,
\]
and set $F_0:=\varnothing$.

For every $z\in V(G)$ and every vertex part $V_t$ with $z\notin V_t$,
let $\mu_s(z,V_t)$ denote the number of indices $j\in[s]$ such that
$z$ is an auxiliary vertex of $R_j$ and both neighbours of $z$ on
$R_j$ lie in $V_t$. Let
$$
M:=\left\lceil\frac{4(2k-2)K}{\gamma}\right\rceil+1.
$$
Throughout the construction, we maintain the invariant
\begin{align}
\mu_s(z,V_t)\leq M
\qquad
\text{for every $z\in V(G)$ and every part $V_t$ with $z\notin V_t$.}
\label{eq:4}
\end{align}
This clearly holds when $s=0$.

We first establish two bounds that will be used in constructing $R_{s+1}$. For any distinct vertex parts $U$ and $W$ of $G$, let
$$
\mathcal S_s(U,W):=
\{z\in U:\mu_s(z,W)=M\}.
$$
Each realization $R_j$ has exactly $2k-2$ auxiliary vertices. Hence
$$
M|\mathcal S_s(U,W)|=\sum_{z\in\mathcal S_s(U,W)}\mu_s(z,W)
\leq \sum_{z\in U}\mu_s(z,W)\leq
(2k-2)s\leq (2k-2)Kn.
$$
By the definition of $M$, we obtain
\begin{align}
|\mathcal S_s(U,W)|<\frac{\gamma}{4}n.
\label{eq:5}
\end{align}
We also have the following bound. For any $w\in V(G)$ and any vertex
part $U$ with $w\notin U$, every auxiliary vertex $w$ in a realization
has both of its neighbours in the same vertex part. Hence, by
\eqref{eq:4}, the realizations in which $w$ is an auxiliary vertex
with both neighbours in $U$ contribute at most $2M$ edges from $w$
to $U$. On the other hand, $w$ is a basic vertex in at most $K'$
members of $\mathcal D$, and each such realization contains at most
one edge from $w$ to $U$. Therefore,
\begin{align}
d_{F_s}(w,U)\leq 2M+K'.
\label{eq:7}
\end{align}

Write $\tau_{s+1}=(p_1;x,y;p_2,p_3)$. We construct a realization $xz_1z_2\cdots z_{2k-3}yz_{2k-2}x$ of $\tau_{s+1}$ by choosing its auxiliary vertices successively. We first choose $z_1,\ldots,z_{2k-4}$, alternating between $V_{p_2}$ and $V_{p_1}$, then choose $z_{2k-3}\in V_{p_2}$ adjacent to both $z_{2k-4}$ and $y$, and finally choose $z_{2k-2}\in V_{p_3}$ adjacent to both $x$ and $y$. Since
$$ d_G(x,V_{p_2})-|\mathcal{S}_s(V_{p_2},V_{p_1})|-d_{F_s}(x,V_{p_2})-d_F(x,V_{p_2})\geq\left(\frac{1}{2}+\gamma\right)n-\frac{\gamma}{4}n-(2M+K')-C>0$$
by taking $n_0$ sufficiently large, we may choose $z_1\in N_G(x,V_{p_2})$ such that $xz_1\notin E(F\cup F_s)$ and $z_1\notin\mathcal{S}_s(V_{p_2},V_{p_1})$. Suppose that $z_1,\ldots,z_q$ have been chosen, where $1\leq q\leq 2k-5$. Define
$$
U_{q+1}:=
\begin{cases}
V_{p_1}, & \text{if } z_q\in V_{p_2},\\
V_{p_2}, & \text{if } z_q\in V_{p_1}.
\end{cases}
$$
By taking $n_0$ sufficiently large, we have
\begin{align*}
&d_G(z_q,U_{q+1})-|\mathcal{S}_s(U_{q+1},U_q)|-d_{F_s}(z_q,U_{q+1})-d_F(z_q,U_{q+1})-2k\\
&\geq\left(\frac{1}{2}+\gamma\right)n-\frac{\gamma}{4}n-(2M+K')-C-2k>0.
\end{align*}
Hence, we may choose $z_{q+1}\in N_G(z_q,U_{q+1})$ such that $z_qz_{q+1}\notin E(F\cup F_s)$, $z_{q+1}\notin\{x,y,z_1,\ldots,z_q\}$, and $z_{q+1}\notin\mathcal{S}_s(U_{q+1},U_q)$. Thus, $z_1,\ldots,z_{2k-4}$ can be chosen successively. Since
$$
|N_G(z_{2k-4},V_{p_2})\cap N_G(y,V_{p_2})|\geq 2\gamma n,
$$
by taking $n_0$ sufficiently large, we have
$$
\begin{aligned}
&|N_G(z_{2k-4},V_{p_2})\cap N_G(y,V_{p_2})|-|\mathcal{S}_s(V_{p_2},V_{p_1})|-d_{F_s}(z_{2k-4},V_{p_2})-d_{F_s}(y,V_{p_2})\\
&\qquad-d_F(z_{2k-4},V_{p_2})-d_F(y,V_{p_2})-2k\geq2\gamma n-\frac{\gamma}{4}n-2(2M+K')-2C-2k>0.
\end{aligned}
$$
Hence, we may choose $z_{2k-3}\in N_G(z_{2k-4},V_{p_2})\cap N_G(y,V_{p_2})$ such that $z_{2k-4}z_{2k-3},z_{2k-3}y\notin E(F\cup F_s)$, $z_{2k-3}\notin\{x,y,z_1,\ldots,z_{2k-4}\}$, and $z_{2k-3}\notin\mathcal{S}_s(V_{p_2},V_{p_1})$. Similarly, we may choose $z_{2k-2}\in N_G(x,V_{p_3})\cap N_G(y,V_{p_3})$ such that $yz_{2k-2},z_{2k-2}x\notin E(F\cup F_s)$, $z_{2k-2}\notin\{x,y,z_1,\ldots,z_{2k-3}\}$, and $z_{2k-2}\notin\mathcal{S}_s(V_{p_3},V_{p_1})$. Hence,
$$ R_{s+1}:=xz_1z_2\cdots z_{2k-3}yz_{2k-2}x $$
is a $C_{2k}$-realization of $\tau_{s+1}$ in $G-(F\cup F_s)$ and is edge-disjoint from $R_1,\ldots,R_s$. By the choice of these vertices, \eqref{eq:4} also holds with $s$ replaced by $s+1$. By induction, all members of $\mathcal D$ admit pairwise edge-disjoint $C_{2k}$-realizations $R_1,\ldots,R_m$. Let
$$ F':=\bigcup_{j=1}^m R_j.$$
By \eqref{eq:7}, with $s=m$, we have
$$ \hat\Delta(F')\leq 2M+K'.$$
Thus, the conclusion follows with $C':=2M+K'$.

When $k=2$, a realization of a tuple $(p_1;x,y;p_2,p_3)$ has the form $xz_1yz_2x$. By the same greedy argument, we may choose $z_1\in N_G(x,V_{p_2})\cap N_G(y,V_{p_2})$ and $z_2\in N_G(x,V_{p_3})\cap N_G(y,V_{p_3})$ successively, while avoiding the forbidden vertices and edges. Hence, the same conclusion follows.
\end{proof}

\begin{Lemma}\label{lem:partite_degree_even}
Let $r\geq3$ and $k\geq2$ be integers, and let $\gamma>0$. Then there exist a constant $C>0$ and an integer $n_0$ such that the following holds. For every integer $n>n_0$, let $G$ be a $C_{2k}$-divisible $r$-partite graph with vertex partition $V_1,\ldots,V_r$ satisfying $n\leq |V_i|\leq 2n$ for every $i\in[r]$ and $\hat{\delta}(G)\geq\frac{1}{2}+\gamma$. If $|E_{G}(V_i,V_j)|\equiv0\pmod2$ for every $1\leq i<j\leq r$, then there exists a subgraph $H_2$ of $G$ such that
\begin{enumerate}
\item[\textup{(i)}] $H_2$ admits a $C_{2k}$-decomposition and $\hat\Delta(H_2)\leq C$;
\item[\textup{(ii)}] $G-H_2$ is $C_{2k}$-divisible;
\item[\textup{(iii)}] $d_{G-H_2}(v,V_t)\equiv0\pmod2$ for every $t\in[r]$ and every $v\in V(G)\setminus V_t$.
\end{enumerate}
\end{Lemma}

\begin{proof}
For every $t\in[r]$, let $P_t$ be the bipartite graph with bipartition $(V_t,[r]\setminus\{t\})$ in which, for $v\in V_t$ and $h\in[r]\setminus\{t\}$,
$vh\in E(P_t)$ if and only if $d_G(v,V_h)\equiv1\pmod2$. For every $v\in V_t$, since $G$ is $C_{2k}$-divisible,
$$ d_{P_t}(v) \equiv \sum_{h\in[r],\, h\neq t} d_{G}(v,V_h) \equiv d_{G}(v) \equiv 0 \pmod 2.$$
For every $h\in[r]\setminus\{t\}$, by assumption,
$$ d_{P_t}(h)\equiv\sum_{v\in V_t}d_{G}(v,V_h)\equiv |E_{G}(V_t,V_h)|\equiv0\pmod2.$$
Thus every vertex of $P_t$ has even degree. Then $P_t$ can be decomposed into pairwise edge-disjoint even cycles. For any even cycle $L=v_1h_1v_2h_2\ldots v_\ell h_\ell v_1$ in this decomposition, where $v_a\in V_t$ and $h_a\in[r]\setminus\{t\}$ for $a=1,2,\ldots,\ell$, introduce, for each $1\leq q\leq \ell-1$, a $4$-cycle
$$
R_q=v_1h_qv_{q+1}h_{q+1}v_1,
$$
and associate with $R_q$ the tuple $(t;v_1,v_{q+1};h_q,h_{q+1})$. Let $\mathcal D_t$ denote the multiset of all tuples arising from the fixed cycle decomposition of $P_t$, and let $\mathcal D$ be the multiset union of the $\mathcal D_t$ over $t\in[r]$.

We now verify that $\mathcal D$ satisfies the hypotheses of Lemma \ref{lem:realization}. Since each such cycle $L$ produces $\ell-1\leq |E(L)|/2$ tuples, $|\mathcal D_t|\leq |E(P_t)|/2$. Since every vertex in $V_t$ has degree at most $r-1$ in $P_t$, we have $|E(P_t)|\leq (r-1)|V_t|$. Thus,
$$|\mathcal D| \leq\sum_{t=1}^r |E(P_t)|/2 \leq \frac{r-1}{2}\sum_{t=1}^r|V_t|\leq \frac{r-1}{2} 2rn = r(r-1)n.$$
For any $v\in V_t$, the edge-disjoint cycle decomposition of $P_t$ contains $v$ in at most $d_{P_t}(v)/2\leq (r-1)/2$ cycles. Each such cycle has length at most $2(r-1)$ and hence contributes at most $r-2$ tuples in which $v$ is a basic vertex. Therefore, every vertex occurs as a basic vertex in at most $\frac{(r-2)(r-1)}{2}$ members of $\mathcal D$. Applying Lemma~\ref{lem:realization} with
$$F=\varnothing,\qquad K=r(r-1),\qquad K'=\frac{(r-2)(r-1)}{2},$$
we obtain pairwise edge-disjoint $C_{2k}$-realizations in $G$, one for each member of $\mathcal D$. For each $t\in[r]$, let $H_{2,t}$ be the union of the realizations corresponding to the tuples in $\mathcal D_t$, and let
$$H_2:=\bigcup_{t=1}^r H_{2,t}.$$
Since $H_2$ is the union of all the realizations corresponding to the tuples in $\mathcal D$, Lemma~\ref{lem:realization} yields a constant $C'$ such that $\hat\Delta(H_2)\leq C'$.

Fix $t\in[r]$, and let $L=v_1h_1v_2h_2\ldots v_\ell h_\ell v_1$ be a cycle in the fixed cycle decomposition of $P_t$. We claim that
$$
E(L)=E(R_1)\triangle\cdots\triangle E(R_{\ell-1}),
$$
where, for $1\leq q\leq \ell-1$, $R_q=v_1h_qv_{q+1}h_{q+1}v_1$ is a $4$-cycle. Indeed, for every $2\leq q\leq \ell-1$, the edge $v_1h_q$ belongs to exactly the two $4$-cycles $R_{q-1}$ and $R_q$, and therefore disappears in the symmetric difference. These are precisely the edges appearing in the $R_q$'s that do not belong to $L$. On the other hand, the edges $v_1h_1$ and $v_1h_\ell$ occur only in $R_1$ and $R_{\ell-1}$, respectively. Furthermore, for every $1\leq q\leq \ell-1$, the edges $h_qv_{q+1}$ and $v_{q+1}h_{q+1}$ occur only in $R_q$. These edges, together with $v_1h_1$ and $v_1h_\ell$, are exactly the edges of $L$. Hence every edge of $L$ occurs exactly once in $E(R_1)\triangle\cdots\triangle E(R_{\ell-1})$, whereas every edge not belonging to $L$ occurs an even number of times. This proves the claim. For each $1\leq q\leq \ell-1$, deleting a $C_{2k}$-realization of the tuple $(t;v_1,v_{q+1};h_q,h_{q+1})$ associated with $R_q=v_1h_qv_{q+1}h_{q+1}v_1$ reverses exactly the parities of the four partite degrees of the vertices $v_1,v_{q+1}\in V_t$ into $V_{h_q}$ and $V_{h_{q+1}}$. Hence, after deleting the $C_{2k}$-realizations of all the tuples associated with $R_1,\ldots,R_{\ell-1}$, for every $v\in V_t$ and every $h\in[r]\setminus\{t\}$, the parity of the partite degree of $v$ into $V_h$ is reversed if and only if $vh\in E(L)$. Since the cycles form an edge decomposition of $P_t$, every edge $vh\in E(P_t)$ is contained in exactly one of these cycles. Therefore, after deleting $H_{2,t}$, for every $v\in V_t$ and every $h\in[r]\setminus\{t\}$, the parity of the partite degree of $v$ into $V_h$ is reversed if and only if $vh\in E(P_t)$. By the definition of $P_t$, this occurs if and only if the partite degree of $v$ into $V_h$ is odd in $G$. Consequently, all these partite degrees are even in $G-H_{2,t}$; that is,
$$
d_{G-H_{2,t}}(v,V_h)\equiv0\pmod2
$$
for every $v\in V_t$ and every $h\in[r]\setminus\{t\}$.

Now fix $u\in[r]\setminus\{t\}$. Deleting $H_{2,u}$ does not change the parity of any partite degree of a vertex in $V_t$. Indeed, the basic vertices of every realization corresponding to a tuple in $\mathcal D_u$ lie in $V_u$. Hence, a vertex $v\in V_t$ can occur in such a realization only as an auxiliary vertex. Whenever this occurs, the two neighbours of $v$ in the realization both lie in $V_u$. Thus, deleting this realization decreases the partite degree of $v$ into $V_u$ by two and leaves all its other partite degrees unchanged. Consequently, deleting $H_{2,u}$ preserves the parity of every partite degree of every vertex in $V_t$. Since this holds for every $u\in[r]\setminus\{t\}$, the even partite degrees obtained after deleting $H_{2,t}$ remain even after all of $H_2$ is deleted. As $t\in[r]$ was arbitrary, every partite degree of $G-H_2$ is even; that is,
$$
d_{G-H_2}(v,V_h)\equiv0\pmod2
$$
for every $h\in[r]$ and every $v\in V(G)\setminus V_h$. In particular, every vertex of $G-H_2$ has even degree, since its degree is the sum of its partite degrees. Since $H_2$ is the union of the pairwise edge-disjoint $C_{2k}$-realizations, it admits a $C_{2k}$-decomposition. In particular, $|E(H_2)|$ is divisible by $2k$, and hence
$$
|E(G-H_2)|=|E(G)|-|E(H_2)|\equiv0\pmod{2k}.
$$
 Thus, $G-H_2$ is $C_{2k}$-divisible. Taking $C:=C'$, the proof is complete.
\end{proof}

Let $\Gamma$ be a connected graph. For $x,y\in V(\Gamma)$, let $d_\Gamma(x,y)$ denote the distance between $x$ and $y$, that is, the length of a shortest path from $x$ to $y$ in $\Gamma$. A \emph{spanning tree} of $\Gamma$ is a spanning subgraph of $\Gamma$ that is a tree. The \emph{line graph} $L(\Gamma)$ of $\Gamma$ has vertex set $E(\Gamma)$, where two vertices are adjacent if and only if the corresponding edges of $\Gamma$ have a common endpoint. In particular, $L(K_{[r]})$ is connected and hence has a spanning tree.

\begin{Lemma}\label{lem:pairwise_C2k_divisible}
Let $r\geq3$ and $k\geq2$ be integers, and let $\gamma>0$. Then there exist a constant $C>0$ and an integer $n_0$ such that the following holds. For every integer $n>n_0$, let $G$ be a $C_{2k}$-divisible $r$-partite graph with vertex partition $V_1,\ldots,V_r$ satisfying $|V_t|\geq n$ for every $t\in[r]$ and $\hat{\delta}(G)\geq\frac{1}{2}+\gamma$. If $d_{G}(v,V_t)\equiv0\pmod2$ for every $t\in[r]$ and every $v\in V(G)\setminus V_t$, then there exists a subgraph $H_3$ of $G$ such that
\begin{enumerate}
\item[\textup{(i)}] $H_3$ admits a $C_{2k}$-decomposition and $\hat\Delta(H_3)\leq C$;
\item[\textup{(ii)}] $(G-H_3)[V_i, V_j]$ is $C_{2k}$-divisible for every $1\leq i<j\leq r$.
\end{enumerate}
\end{Lemma}

\begin{proof}
Let $T$ be a spanning tree of the line graph $L(K_{[r]})$. We identify each vertex of $T$ with the corresponding $2$-subset of $[r]$, and write $ij$ for the vertex corresponding to the edge $\{i,j\}$ of $K_{[r]}$. Fix a vertex $i_0j_0\in V(T)$. Let $m:=\max_{ij\in V(T)}d_T(ij,i_0j_0)$, and, for $s=0,1,\ldots,m$, let
$$
L_s:=\{ij\in V(T):d_T(ij,i_0j_0)=s\}.
$$
Then the nonempty sets $L_0,L_1,\ldots,L_m$ form a partition of $V(T)$.

We construct pairwise edge-disjoint subgraphs $Q_m,Q_{m-1},\ldots,Q_1$ of $G$ successively in this order; their union will be $H_3$. Suppose that, for some $1\leq s\leq m$, the subgraphs $Q_m,\ldots,Q_{s+1}$ have already been constructed, with no subgraph constructed yet when $s=m$. Let
$$
F_s:=\bigcup_{\ell=s+1}^{m}Q_\ell, \quad \text{and} \quad \mathcal{L}_s:=\bigcup_{\ell=s+1}^{m}L_\ell,
$$
where $F_s$ is an empty graph and $\mathcal{L}_s=\varnothing$ if $s=m$. Let $G^s:=G-F_s$. We maintain that every partite degree of $G^s$ is even, that
$|E_{G^s}(V_i,V_j)|\equiv0\pmod{2k}$ for every $ij\in\mathcal{L}_s$, that $|E(G^s)|\equiv0\pmod{2k}$, and that $\hat\Delta(F_s)\leq C_s$, where $C_s$ is a constant independent of $n$. These properties hold when $s=m$, since $G^m=G$ has even partite degrees and $|E(G^m)|\equiv0\pmod{2k}$ by hypothesis, $\mathcal{L}_m=\varnothing$, and $F_m$ is an empty graph of $G$; we may therefore take $C_m:=0$.

For each vertex in $L_s$, write it as $ij$ and its unique neighbour in $L_{s-1}$ as $ih$, where $i$ is their common endpoint. Since every partite degree of $G^s$ is even, $|E_{G^s}(V_i,V_j)|$ is even. Let $\lambda_{ij}$ be the unique even integer with $0\leq\lambda_{ij}\leq2k-2$ such that $ |E_{G^s}(V_i,V_j)|-\lambda_{ij}\equiv0\pmod{2k}$. Choose distinct vertices $x_1,x_2,y_1,y_2,y_3\in V_i$, depending on $ij$. Let $\mathcal D_{(i,j,h)}$ be a multiset of tuples defined as follows. If $\lambda_{ij}\equiv0\pmod4$, take $\lambda_{ij}/2$ copies of the tuple $(i;x_1,x_2;h,j)$. If $\lambda_{ij}\equiv2\pmod4$ and $\lambda_{ij}>2$, take the three tuples $(i;y_1,y_2;h,j), (i;y_2,y_3;h,j), (i;y_3,y_1;h,j)$, together with $\lambda_{ij}/2-3$ copies of the tuple $(i;x_1,x_2;h,j)$. If $\lambda_{ij}=2$, take the three tuples $(i;y_1,y_2;h,j), (i;y_2,y_3;h,j), (i;y_3,y_1;h,j)$, together with two copies of the tuple $(i;x_1,x_2;j,h)$. Let $\mathcal D_s$ be the multiset union of $\mathcal D_{(i,j,h)}$ over all $ij\in L_s$. For each $ij\in L_s$, the number of tuples in $\mathcal D_{(i,j,h)}$ is less than $3k$. Hence
$$ |\mathcal D_s|\leq3k|L_s|\leq3k\binom{r}{2},$$
and every vertex of $G$ occurs as a basic vertex in at most $3k\binom{r}{2}$ members of $\mathcal D_s$. By choosing $n_0$ sufficiently large, we may apply Lemma \ref{lem:realization} at each step. Applying Lemma~\ref{lem:realization} with
$$ K=K'=3k\binom{r}{2},\qquad F=F_s\quad (\hat\Delta(F_s)\leq C_s),$$
we obtain pairwise edge-disjoint $C_{2k}$-realizations in $G-F_s=G^s$, one for each member of $\mathcal D_s$. Let $Q_s$ be the union of these realizations. By Lemma \ref{lem:realization}, there exists a constant $C'_s$ such that $\hat\Delta(Q_s)\leq C'_s$.

A $C_{2k}$-realization of a tuple $(i;x,y;h,j)$ contains $2k-2$ edges between $V_i$ and $V_h$ and two edges between $V_i$ and $V_j$, whereas a realization of $(i;x,y;j,h)$ contains two edges between $V_i$ and $V_h$ and $2k-2$ edges between $V_i$ and $V_j$. Hence, by the definition of $\mathcal D_{(i,j,h)}$,
$$|E_{Q_s}(V_i,V_j)|\equiv\lambda_{ij}\pmod{2k}$$
for every $ij\in L_s$. Moreover, every partite degree of $Q_s$ is even. Indeed, every auxiliary vertex has two neighbours in the same vertex part. For the basic vertices, if $\lambda_{ij}\equiv0\pmod4$, then each of $x_1,x_2$ occurs as a basic vertex an even number of times. If $\lambda_{ij}\equiv2\pmod4$ and $\lambda_{ij}>2$, then each of $y_1,y_2,y_3$ occurs twice, while each of $x_1,x_2$ occurs $\lambda_{ij}/2-3$ times, which is even. If $\lambda_{ij}=2$, then each of $y_1,y_2,y_3$ occurs twice and each of $x_1,x_2$ occurs twice. Thus every basic vertex also has even partite degrees in $Q_s$. Furthermore, every realization used in the construction of $Q_s$ has edges only in the two part-pairs corresponding to a vertex $ij\in L_s$ and its unique neighbour $ih\in L_{s-1}$. Therefore, $Q_s$ has no edge in any part-pair corresponding to a vertex in $L_m\cup\cdots\cup L_{s+1}$.

Let $G^{s-1}:=G^s-Q_s$, $F_{s-1}:=F_s\cup Q_s$, $C_{s-1}:=C_s+C'_s$. Then $\hat\Delta(F_{s-1})\leq C_s+C'_s=C_{s-1}$. Since every partite degree of both $G^s$ and $Q_s$ is even, every partite degree of $G^{s-1}$ is even. For every $ij\in L_s$, by the definition of $\lambda_{ij}$ and the construction of $Q_s$, we have
$$ |E_{G^{s-1}}(V_i,V_j)| = |E_{G^s}(V_i,V_j)|-|E_{Q_s}(V_i,V_j)| \equiv0\pmod{2k}.$$
Moreover, $Q_s$ has no edge in any part-pair corresponding to a vertex in $\mathcal L_s=L_m\cup\cdots\cup L_{s+1}$. Hence the number of edges of $G^{s-1}$ in every such part-pair is unchanged and remains divisible by $2k$. Since $\mathcal L_{s-1}=L_s\cup\mathcal L_s$, we have $|E_{G^{s-1}}(V_i,V_j)|\equiv0\pmod{2k}$ for every $ij\in\mathcal L_{s-1}$. Finally, since $Q_s$ is a union of edge-disjoint copies of $C_{2k}$,
$$ |E(G^{s-1})| = |E(G^s)|-|E(Q_s)| \equiv0\pmod{2k}. $$
Thus, all the required properties are maintained with $s$ replaced by $s-1$.

Repeating this argument for $s=m,m-1,\ldots,1$, we obtain pairwise edge-disjoint subgraphs $Q_m,Q_{m-1},\ldots,Q_1$. Let $H_3:=\bigcup_{s=1}^{m}Q_s$ and $G^0:=G-H_3$. Then every partite degree of $G^0$ is even, $|E(G^0)|\equiv0\pmod{2k}$, and $|E_{G^0}(V_i,V_j)| \equiv 0 \pmod{2k}$ for every $ij\in L_m\cup\cdots\cup L_1$. It remains to consider the part-pair corresponding to the vertex $i_0j_0$ in $L_0$. Since $H_3$ is a union of edge-disjoint copies of $C_{2k}$, the graph $H_3$ admits a $C_{2k}$-decomposition. Since the number of edges in every part-pair other than $(V_{i_0},V_{j_0})$ is divisible by $2k$ and $|E(G^0)|\equiv0\pmod{2k}$, the number of edges between $V_{i_0}$ and $V_{j_0}$ is also divisible by $2k$. Therefore, $|E_{G^0}(V_i,V_j)| \equiv0 \pmod{2k}$ for every $1\leq i<j\leq r$. Since every partite degree of $G^0$ is even, $G^0[V_i, V_j]$ is $C_{2k}$-divisible for every $1\leq i<j\leq r$. Finally, $H_3=F_0$, and thus $\hat\Delta(H_3)=\hat\Delta(F_0)\leq C_0$. Taking $C:=C_0$, both \textup{(i)} and \textup{(ii)} follow.
\end{proof}

\subsection{Proof of Theorem \ref{thm:main}}\label{sec:main}

\begin{proof}[\bf Proof of Theorem \ref{thm:main}]
By Proposition \ref{prop:lower}, we have $\hat\delta^{r}_{C_{2k}} \geq \alpha_k$, so it remains to prove $\hat\delta^{r}_{C_{2k}} \leq \alpha_k$. Let $\varepsilon>0$. If $r=2$, the result follows directly from Theorem \ref{thm:Taylor}. Hence, we may assume that $r\geq3$. Set $\gamma:=\alpha_k-\frac{1}{2}+\frac{\varepsilon}{2}$, so that $\gamma>0$ and $\frac12+\gamma=\alpha_k+\frac{\varepsilon}{2}$. Apply Lemmas \ref{lem:part_pair_even}, \ref{lem:partite_degree_even}, and \ref{lem:pairwise_C2k_divisible} with this value of $\gamma$, and denote the resulting constants by $(C_1,n_1)$, $(C_2,n_2)$, and $(C_3,n_3)$, respectively, where $C_1=2$. By Theorem \ref{thm:Taylor}, there exists an integer $n_4$ such that every $C_{2k}$-divisible $2$-balanced bipartite graph $G'$ with minimum vertex-part size greater than $n_4$ and $\hat{\delta}(G')>\alpha_k+\frac{\varepsilon}{2}$ admits a $C_{2k}$-decomposition. Choose $n_0$ sufficiently large so that
$$
n_0>\max\{n_1,n_2,n_3,n_4\}
\quad\text{and}\quad
\frac{C_1+C_2+C_3}{n_0}<\frac{\varepsilon}{2}.
$$

Let $n>n_0$, and let $G$ be a $C_{2k}$-divisible $r$-partite graph with vertex partition $V_1,\ldots,V_r$ such that $n\leq |V_t|\leq 2n$ for every $t\in[r]$, and $\hat{\delta}(G)\geq\alpha_k+\varepsilon$. Since $\alpha_k+\varepsilon=\frac{1}{2}+\gamma+\frac{\varepsilon}{2}$, we have $\hat{\delta}(G)\geq\frac{1}{2}+\gamma+\frac{\varepsilon}{2}$. By Lemma~\ref{lem:part_pair_even}, there exists a subgraph $H_1$ of $G$ such that $H_1$ admits a $C_{2k}$-decomposition, $\hat\Delta(H_1)\leq C_1$, the graph $G_1:=G-H_1$ is $C_{2k}$-divisible, and $|E_{G_1}(V_i,V_j)|\equiv0\pmod2$ for every $1\leq i<j\leq r$. Moreover,
$$\hat{\delta}(G_1)\geq\hat{\delta}(G)-\frac{C_1}{n} > \frac{1}{2}+\gamma.$$
Hence, by Lemma~\ref{lem:partite_degree_even}, there exists a subgraph $H_2$ of $G_1$ such that $H_2$ admits a $C_{2k}$-decomposition, $\hat\Delta(H_2)\leq C_2$, the graph $G_2:=G_1-H_2$ is $C_{2k}$-divisible, and $d_{G_2}(v,V_t)\equiv0\pmod2$ for every $t\in[r]$ and every $v\in V(G_2)\setminus V_t$. Furthermore,
$$\hat{\delta}(G_2)\geq\hat{\delta}(G)-\frac{C_1+C_2}{n}>\frac{1}{2}+\gamma.$$
Therefore, by Lemma~\ref{lem:pairwise_C2k_divisible}, there exists a subgraph $H_3$ of $G_2$ such that $H_3$ admits a $C_{2k}$-decomposition, $\hat\Delta(H_3)\leq C_3$, and, letting $G_3:=G_2-H_3$, the bipartite graph $G_3[V_i, V_j]$ is $C_{2k}$-divisible for every $1\leq i<j\leq r$. Furthermore,
$$ \hat{\delta}(G_3) \geq \hat{\delta}(G)-\frac{C_1+C_2+C_3}{n} > \alpha_k+\frac{\varepsilon}{2}. $$
For every $1\leq i<j\leq r$, the graph $G_3[V_i,V_j]$ is $2$-balanced, has minimum vertex-part size at least $n>n_4$, and satisfies $\hat{\delta}(G_3[V_i,V_j])\geq\hat{\delta}(G_3)>\alpha_k+\frac{\varepsilon}{2}$. Thus, by Theorem \ref{thm:Taylor}, $G_3[V_i, V_j]$ admits a $C_{2k}$-decomposition for every $1\leq i<j\leq r$.

Since the edge sets of $H_1,H_2,H_3$ and of $G_3[V_i,V_j]$, $1\leq i<j\leq r$, form a partition of $E(G)$, their $C_{2k}$-decompositions together give a $C_{2k}$-decomposition of $G$.
\end{proof}

\section{Concluding remarks}

Let $k\geq2$ be an integer. A \emph{group divisible design} is a triple $(X,\mathcal G,\mathcal B)$, where $X$ is a finite set of points, $\mathcal G$ is a partition of $X$ into subsets called \emph{groups}, and $\mathcal B$ is a collection of $k$-subsets of $X$ (called \emph{blocks}), such that every pair of distinct points of $X$ is contained either in exactly one group or in exactly one block, but not both. If $\mathcal G$ consists of $a_i$ groups of size $n_i$ for each $i=1,\ldots,s$, then such a design is called a $k$-GDD of \emph{type} $n_1^{a_1}n_2^{a_2}\cdots n_s^{a_s}$. Equivalently, a $k$-GDD of type $n_1^{a_1}n_2^{a_2}\cdots n_s^{a_s}$ is a $k$-clique decomposition of the complete multipartite graph whose vertex set is partitioned into $a_i$ parts of size $n_i$ for each $i=1,\ldots,s$. The existence theory of $k$-GDDs is much more developed for uniform group types than for nonuniform group types. The analogous notion for cycles is a $C_k$-GDD, which is a $C_k$-decomposition of the complete multipartite graph (see \cite{CB,CNT,WJC}). Theorem \ref{thm:main} provides an asymptotic decomposition result for dense partial $C_{2k}$-GDDs: it determines the minimum partite-degree threshold that guarantees a $C_{2k}$-decomposition of an $r$-partite graph, provided that the graph is $C_{2k}$-divisible. In particular, the result applies to nonuniform group sizes, with $n\leq n_i\leq2n$ for every $i\in[r]$.

Taylor~\cite{Taylor} proved that, for sufficiently large $n$, every $C_4$-divisible graph on $n$ vertices with minimum degree at least $\frac{2}{3}n-1$ admits a $C_4$-decomposition, and, for every $k\geq4$,
every sufficiently large $C_{2k}$-divisible graph with minimum degree at least $\frac{1}{2}n$ admits a $C_{2k}$-decomposition. These results give exact minimum-degree thresholds, whereas Theorem~\ref{thm:main} determines the corresponding partite-degree threshold only asymptotically. It is therefore natural to ask whether the latter can likewise be made exact, that is, whether the $\varepsilon$ in the minimum partite-degree condition can be removed.

In the bipartite setting, Cavenagh \cite{Cavenagh} proved that if $G$ is a $C_4$-divisible bipartite graph with vertex parts $X$ and $Y$, where $|Y|$ is even and every vertex of $Y$ has degree at least $\frac{95}{96}|X|$, then $G$ admits a $C_4$-decomposition. This result does not require the two vertex parts to have comparable sizes and imposes no corresponding lower bound on the degrees of the vertices in $X$. It would be interesting to extend this result to $r$-partite graphs.

\section*{Declaration of competing interest}

The authors declare that they have no known competing financial interests or personal relationships that could have appeared to influence the work reported in this paper.

\section*{Data availability}

Data sharing is not applicable to this article as no datasets were generated or analyzed during the current study.

\end{document}